\documentclass[a4paper,10pt]{article}
\usepackage{amsmath,amssymb,amsthm,amscd}
\usepackage[mathscr]{eucal}
\usepackage[all]{xy}
\usepackage{graphicx}

\def\bn{\bigskip\noindent}

\makeatletter
  
  \@addtoreset{equation}{section}
\makeatother

\newcommand{\plim}{\varprojlim}
\newcommand{\mcal}{\mathcal}
\newcommand{\mbf}{\mathbf}

\newcommand{\mbb}{\mathbb}
\newcommand{\mrm}{\mathrm}

\newcommand{\divi}{{\rm \mathchar`- div}}
\def\div{_\mathrm{div}}
\def\Gm{\mathbb{G}_\mathrm{m}}
\def\GK{G_K}
\def\Im{\mathrm{Im}}
\def\kgm{k_{g,\mbf{m}}}
\def\Mdiv{M_\mathrm{div}}
\def\Melldiv{M_{\ell\text{\rm{-div}}}}
\def\mulinf{\mu_{\ell^\infty}}
\def\Q{\mathbb{Q}}
\def\Ql{\Q_\ell}
\def\Qp{\Q_p}
\def\tor{_\mathrm{tor}}
\def\Z{\mathbb{Z}}
\def\Zl{{\Z_\ell}}
\def\Zp{{\Z_p}}
\newtheorem{theorem}{Theorem}[section]
\newtheorem{corollary}[theorem]{Corollary}
\newtheorem{lemma}[theorem]{Lemma}
\newtheorem{proposition}[theorem]{Proposition}

\theoremstyle{definition}
\newtheorem{definition}[theorem]{Definition}
\newtheorem{remark}[theorem]{Remark}

\newtheorem{example}[theorem]{Example}

\newtheorem*{question}{Question}

\title{A note on highly Kummer-faithful fields}

\author{Yoshiyasu Ozeki\footnote{
Department of Mathematics and Physics, 
Faculty of Science, Kanagawa University,
2946 Tsuchiya, Hiratsuka-shi, Kanagawa 259-1293, JAPAN
\endgraf
e-mail: {\tt ozeki@kanagawa-u.ac.jp}} 
and 
Yuichiro Taguchi\footnote{
Tokyo Institute of Technology, 
2-12-1 Ookayama, Meguro-ku, Tokyo 152-8551, JAPAN
\endgraf
e-mail: {\tt taguchi@math.titech.ac.jp}}
}

\begin{document}
\maketitle

\begin{abstract}
We introduce a notion of {\it highly Kummer-faithful fields} and 
study its relationship with the notion of Kummer-faithful fields.
We also give some examples of  highly Kummer-faithful fields.
For example, if  $k$  is a number field of finite degree over  $\Q$, 
$g$  is an integer $>0$ and 
$\mbf{m}=(m_p)_p$  is a family of non-negative integers, 
where $p$ ranges over all prime numbers,
then the extension field 
$k_{g,\mbf{m}}$  obtained by adjoining to  $k$  
all coordinates of the elements of the $p^{m_p}$-torsion subgroup    
$A[p^{m_p}]$  of  $A$  for all semi-abelian varieties  
$A$  over  $k$  of dimension at most  $g$  
and all prime numbers $p$, is highly Kummer-faithful. 
\end{abstract}

\section{Introduction}

The notion of a Kummer-faithful field, 
defined by Mochizuki \cite{Mo}, 
plays an important role in anabelian geometry; 
it is defined by the non-divisibility of 
the Mordell-Weil groups of semi-abelian varieties over the field 
(see Definition \ref{Def:KF}), 
so that Kummer theory works effectively for the purpose of 
reconstructing geometric objects over such a field.  
Recently, 
some versions of the  Grothendieck Conjecture in anabeian geometry are 
being generalized from ones over sub-$p$-adic fields to 
ones over Kummer-faithful fields (see, e.g., \cite{Ho}). 
Note that a sub-$p$-adic field (for a prime number $p$) 
is Kummer-faithful (\cite[Remark 1.5.4]{Mo}). 
In particular,  
finite extensions of $\Q(t_1,\dots ,t_r)$  and  $\Qp(t_1,\dots ,t_r)$ 
are Kummer-faithful.
On the other hand, 
many of the infinite algebraic extensions of  $\Q$  are not Kummer-faithful.
For example, 
the field  $\Q(\mu_{p^{\infty}})$  generated over $\Q$
by all $p$-power roots of unity is not Kummer-faithful.
In general, it is not easy to determine 
whether a given field is Kummer-faithful or not. 
In this paper, 
we give in \S2 a sufficient condition for a {\it Galois} extension of 
a finite number field to be 
Kummer-faithful in terms of ramification theory 
(Theorem \ref{Main:HKF} and Corollary \ref{Main:HKF:cor}), 
and in 3 construct in \S3 examples of 
Kummer-faithful fields of infinite degree over  $\Q$
which are not sub-$p$-adic for any prime number  $p$. 
Some similar examples have already been studied in Ohtani \cite{Oh};
we give a more systematic treatment in this paper. 

Furthermore,  we define a closely related notion 
``highly Kummer-faithful field" 
(Definition \ref{Def:HKF})
and study its relation to Kummer-faithfulness. 
High Kummer-faithfulness is defined by the vanishing of the 
coinvariant spaces of the \'etale cohomology groups (with Tate twists) of 
proper smooth varieties over the field.
Precisely speaking, 
we say that a perfect field $K$ of characteristic $p_K$ is highly Kummer-faithful
if, for every finite extension $L$ of $K$ and 
every proper smooth 
variety $X$ over $L$, it holds that 
$$
H^i_{\rm \acute{e}t}(X_{\overline{K}},\mbb{Q}_{\ell}(r))_{G_L}=0\quad 
\text{for any prime number  $\ell\not= p_K$  and  
any  $i, r$  with  $i\not=2r$.}
$$ 
Although not a priori from the definition, 
it follows that,  under a certain condition,
high Kummer-faithfulness implies Kummer-faithfulness 
(Proposition \ref{KFHKF}). 
The above mentioned sufficient condition in \S2 applies in fact 
also to high Kummer-faithfulness, and our examples of fields 
in \S3 are highly Kummer-faithful.

\vspace{5mm}
\noindent
{\bf Notation.}
A {\it number field} is a finite 
extension of the field  $\mbb{Q}$ of rational numbers.
A {\it $p$-adic field} is a finite extension of the field  $\mbb{Q}_p$ of $p$-adic numbers.
For any field $F$, 
we fix a separable closure $\overline{F}$ of $F$
and we  denote by $G_F$
the absolute Galois group $\mrm{Gal}(\overline{F}/F)$ of $F$. 
If a group  $G$  acts on a vector space  $H$,  
we denote by  $H^G$  (resp.\  $H_G$)  the invariants (resp.\ coinvariants) 
of  $H$  by  $G$; thus   $H^G$  (resp.\  $H_G$)   is the maximal subspace 
(resp.\ quotient space) of  $H$  on which  $G$  acts trivially.

%

\section{Highly Kummer-faithful fields}

In this section, we recall the definition of  Kummer-faithful fields 
(cf.\ \cite[Def.\ 1.5]{Mo}, \cite[Def.\ 1.2]{Ho})
and define a notion of highly Kummer-faithful fields. 
We also study the relationship between Kummer-faithfulness, 
high Kummer-faithfulness and some other properties.

\subsection{Definitions}

\begin{definition} 
Let $M$ be a $\mbb{Z}$-module and $\ell$ a prime number. 
We say that 
$P\in M$ is {\it divisible} (resp.\ {\it $\ell$-divisible}) 
if, for any integer $n>0$, 
there exists $Q\in M$ such that $P=nQ$ (resp.\ $P=\ell^nQ$).  
We denote by $M_{\mrm{div}}$ (resp.\ $M_{\ell\divi}$)
the set of divisible (resp.\ $\ell$-divisible) elements of $M$, that is,
$$
M_{\mrm{div}}=\bigcap_{n>0} n M,     \qquad
M_{\ell\divi}=\bigcap_{n>0}\ell^n M.
$$ 
\end{definition}
Note that 
$P\in M$  is divisible if and only if it is $\ell$-divisible for all primes $\ell$; 
we have  
$\Mdiv=\bigcap_{\ell}\Melldiv$.

\begin{definition} 
\label{Def:KF}
A perfect field $K$ is {\it Kummer-faithful} 
if, for every finite extension $L$ of $K$ and every semi-abelian variety 
$A$ over $L$, it holds that 
$$
A(L)_{\rm div} =0.
$$ 
\end{definition}

The facts below immediately follow from the definition of
Kummer-faithfulness.

\begin{itemize}
\item[(1)] 
If a perfect field  $K$  is Kummer-faithful,
then so is any subfield of  $K$.
\item[(2)] 
Let  $K'$  be a finite extension of a perfect field  $K$.
Then  $K$  is Kummer-faithful 
if and only if  $K'$  is Kummer-faithful.
\end{itemize}

Note that 
a perfect field $K$ is  Kummer-faithful 
if and only if $A(K)_{\rm div} =0$ 
for every semi-abelian variety $A$ over $K$ 
(cf.\ \cite[Rem.\ 1.5.2]{Mo}).

It is also shown in \cite[Rem.\ 1.5.4]{Mo} that 
any sub-$p$-adic field is Kummer-faithful. 
Here, a field $k$ is {\it sub-$p$-adic} if there exists a prime number $p$  and 
a finitely generated field extension $L$ of $\mbb{Q}_p$ such that 
$k$ is isomorphic to a subfield of $L$. 
In particular, a number field is Kummer-faithful. 

\begin{proposition} 
\label{KFreduction}
A perfect field 
$K$ is Kummer-faithful if and only if
$\Gm(L)\div=0$  for any finite extension  $L$  of  $K$  and  
$A(K)\div=0$  for any abelian variety  $A$  over  $K$.
\end{proposition}

\begin{proof}
It suffices to show the ``if'' direction.
Let $B$ be a semi-abelian variety over $K$; thus 
$B$ is an extension of an abelian variety $A$ 
by a torus $T$. It is enough to show  $B(K)_{\rm {div}}=0$
(see just after Definition \ref{Def:KF}).
Assume that there exists a non-zero divisible element $P\in B(K)_{\rm {div}}$.
For any integer $n>0$, let $X_n$ be the set of $P'\in B(K)$ such that $nP'=P$. 
The set  $X_n$ is a non-empty finite set.
Furthermore,  $\{X_n\}_n$  forms a projective system with transition maps  
$f_{n,m}\colon X_m\to X_n$  given by  $f_{n,m}(P')=(m/n)P'$ for $n\mid m$. 
Then the projective limit  $\plim_{n} X_n$  is non-empty 
by \cite[Chap.\ 1, \S 9.6, Prop.\ 8]{Bo}. 
Take  any element $(P_n)_n\in \plim_n X_n$. 
Then the image  $\bar{P}_n$  of  $P_n$  in  $A(K)$  is divisible.
Since  $A(K)_{\rm div} =0$  by assumption, 
$\bar{P}_n$  is zero for any  $n$.
Therefore,  $P_1$ is a non-zero divisible element of  $T(K)$. 
But this is a contradiction, since  $T(K)\div=0$. 
Indeed, if the torus  $T$  splits by a finite extension  $L/K$, then 
$T(K)$  is identified a subgroup of  $\Gm(L)^{\oplus d}$  with  $d=\dim(T)$ and hence  
$T(K)\div=0$.
\end{proof}

For a semi-abelian variety $A$ over a field $k$ and a prime number $\ell$,
we set $T_{\ell}(A):=\plim_{n} A[\ell^n]$  and 
$V_{\ell}(A):=\mbb{Q}_{\ell}\otimes_{\mbb{Z}_{\ell}} T_{\ell}(A)$, where 
the transition map  
$A[\ell^{n+1}]\to A[\ell^n]$  is given by the multiplication-by-$\ell$.
We call  $T_{\ell}(A)$  and  $V_{\ell}(A)$ 
respectively the {\it $\ell$-adic} and {\it rational $\ell$-adic Tate modules} of $A$. 
It is well-known that $T_{\ell}(A)$ is a free $\mbb{Z}_{\ell}$-module of finite rank 
and  
$V_{\ell}(A)$ is a finite dimensional $\mbb{Q}_{\ell}$-vector space.
The Galois group $G_k$ acts on $T_{\ell}(A)$ and $V_{\ell}(A)$ continuously.

\begin{proposition} 
\label{div}
Let $A$ be a semi-abelian variety over a field $k$ and $K$ an algebraic extension of $k$.

\noindent
{\rm (1)} For any prime number $\ell$, 
the following conditions are equivalent.
\begin{itemize}
\item[{\rm (i)}] $(A(K)[\ell^{\infty}])_{\ell\divi}$ is zero.
\item[{\rm (ii)}] $A(K)[\ell^{\infty}]$ is finite.
\item[{\rm (iii)}] $V_{\ell}(A)^{G_K}=0$.
\end{itemize}

\noindent
{\rm (2)} Consider the following conditions.
\begin{itemize}
\item[{\rm (a)}] $A(K)_{\rm div}$ is zero.
\item[{\rm (b)}] $(A(K)_{\rm tor})_{{\rm div}}$ is zero.
\item[{\rm (c)}] $A(K)[\ell^{\infty}]$ is finite for any prime number $\ell$.
\end{itemize}
Then we have $(a)\Rightarrow (b)\Leftrightarrow (c)$. 
If $k$ is  Kummer-faithful and
the extension $K/k$ is Galois, then 
we have $(a)\Leftrightarrow (b)\Leftrightarrow (c)$.  
\end{proposition}

\begin{remark}
For the implication (b) $\Rightarrow$ (a), 
the assumption that  $K/k$  be Galois is essential, as the following example shows: 
let  $k=\Q$  and fix an integer  $a>1$. 
Let  $K$  be the extension of  $k$  obtained by adjoining 
the real $n$th roots of  $a$  for all integers  $n\geq 1$. Then  
$\Gm(K)\div\not=0$, whereas  
$(\Gm(K)\tor)\div=0$.
\end{remark}

\begin{proof}
(1) 
The equivalence of (ii) and (iii) follows immediately from the definition of Tate modules.
We also have a natural isomorphism 
$T_{\ell}(A)^{G_K}\simeq \mrm{Hom}_{\mbb{Z}}(\mbb{Z}[1/\ell]/\mbb{Z},A(K)[\ell^{\infty}])$,
which implies the equivalence of (i) and (iii).

\noindent
(2) The implication $(a)\Rightarrow (b)$ is clear.
The implication $(b)\Rightarrow (c)$ follows from the natural isomorphism 
$\prod_{\ell} T_{\ell}(A)^{G_K}\simeq 
\mrm{Hom}_{\mbb{Z}}(\mbb{Q}/\mbb{Z},A(K)_{\rm tor})$, 
where $\ell$ ranges over the prime numbers.
To show $(c)\Rightarrow (b)$, assume that 
a non-zero divisible element $P$ of $A(K)_{\rm tor}$ exists.
Let $N$ be the order of $P$ and take a prime divisor $\ell$ of $N$.
Then we see that 
$(\ell^{-1}N)P$ is a non-zero $\ell$-divisible element of $A(K)[\ell^{\infty}]$.
By (1), we have that $A(K)[\ell^{\infty}]$ is infinite. 
This shows $(c)\Rightarrow (b)$.

Now we show $(b)\Rightarrow (a)$
under the assumption that $k$ is Kummer-faithful and 
 $K/k$ is Galois. 
Assume that there exists a non-zero divisible element $P$ of $A(K)$.
For any integer $n>0$, 
let $X_n$ be the set of $P'\in A(K)$ such that $P=nP'$.
As is explained in the proof of Proposition \ref{KFreduction}, 
we know that  $\{X_n \}_n$  forms a projective system and its projective limit 
$\plim_n X_n$ is non-empty.  
Take an element $(P_n)_n$ of $\plim_{n} X_n$.
Let $k'/k$ be a  finite subextension of $K/k$ such that $P\in A(k')$.
Since $k$ is Kummer-faithful, $k'$ is also Kummer-faithful.
Hence we have $P_{n_0}\notin A(k')$ for some $n_0$.
Take $\sigma_0\in G_{k'}$ such that $\sigma_0(P_{n_0})-P_{n_0}\not=0$ and
set $Q_n:=\sigma_0(P_n)-P_n$ for $n>0$. Then 
$Q_{n_0}\not=0$ and $nQ_{nn_0}=Q_{n_0}$ for any $n$.
Since the extension $K/k$ is Galois, we have 
$Q_n\in A(K)[n]$  for any $n$.
Hence $Q_{n_0}$ is a non-zero divisible element of $A(K)_{\rm tor}$,
which shows $(b)\Rightarrow (a)$.
\end{proof}

\begin{definition} 
\label{Def:HKF}
Let $K$ be a perfect field with characteristic $p_K\ge 0$. 

\noindent
(1) We say that  $K$ is {\it quasi-highly Kummer-faithful} 
if, for every finite extension $L$ of $K$ and 
every proper smooth 
variety $X$ over $L$, it holds that 
$$
(\ast)\quad  H^i_{\rm \acute{e}t}(X_{\overline{K}},\mbb{Q}_{\ell})_{G_L}=0\quad
 \text{for any prime number  $\ell\not= p_K$  and 
any  $0<i\le 2\dim X$}.
$$ 

\noindent
(2)
We say that  $K$ is {\it  highly Kummer-faithful} 
if, for every finite extension $L$ of $K$ and 
every proper smooth 
variety $X$ over $L$, it holds that 
$$
(\sharp)\quad  H^i_{\rm \acute{e}t}(X_{\overline{K}},\mbb{Q}_{\ell}(r))_{G_L}=0\quad 
\text{for any prime number  $\ell\not= p_K$  and  
any  $i, r$  with  $i\not=2r$.}
$$ 
\end{definition}

We have the following implication: 
$$
\mbox{$K$ is highly Kummer-faithful}
\Rightarrow 
\mbox{$K$ is  quasi-highly Kummer-faithful}.
$$

Note that the condition $(\sharp)$ is equivalent to the condition
$$
(\sharp)' \quad  H^i_{\rm \acute{e}t}(X_{\overline{K}},\mbb{Q}_{\ell}(r))^{G_L}=0\quad {\rm where}\
\ell\not= p_K \ {\rm is\ any\ prime\ number\ and\ }i\not=2r.
$$ 
Indeed, we have isomorphisms
$(H^i_{\rm \acute{e}t}(X_{\overline{K}},\mbb{Q}_{\ell}(r))_{G_L})^{\vee}
\simeq 
(H^i_{\rm \acute{e}t}(X_{\overline{K}},\mbb{Q}_{\ell}(r))^{\vee})^{G_L}
\simeq 
H^{2g-i}_{\rm \acute{e}t}(X_{\overline{K}},\mbb{Q}_{\ell}(g-r))^{G_L}$, 
where $g=\dim X$. 

We also note that the  facts below immediately follow from the definition of
high Kummer-faithfulness (resp.\ quasi-high Kummer-faithfulness).

\begin{itemize}
\item[(1)] If $K$ is highly Kummer-faithful (resp.\ quasi-highly Kummer-faithful),
then so is any subfield of $K$.
\item[(2)] Let $K'$ be a finite extension of a perfect field $K$.
Then $K$ is highly Kummer-faithful (resp.\ quasi-highly Kummer-faithful) 
if and only if 
$K'$ is highly Kummer-faithful (resp.\ quasi-highly Kummer-faithful).
\end{itemize}

\begin{remark} 
In view of the fact that Kummer theory deals with 
abelian extensions of degree prime to the characteristic of the base field, 
we impose the condition  $\ell\not= p_K$  in the above definition of 
(quasi-)high Kummer-faithfulness, 
though there may be room for discussion whether we 
should really do so.
\end{remark}

\begin{remark} 
We should note that 
Kummer-faithfulness and (quasi-)high Kummer-faithfulness are
``local notion'' in coefficients. 
In fact, it is not difficult to check that 
a perfect field $K$ is Kummer-faithful
if and only if $K$ is {\it $\ell$-Kummer-faithful} for any prime number $\ell$. 
Here,
$K$ is $\ell$-Kummer-faithful
if, for every finite extension $L$ of $K$ and every semi-abelian variety 
$A$ over $L$, it holds that $A(L)_{\ell\divi} =0.$
If we define the notion of {\it high $\ell$-Kummer-faithfulness} 
by the obvious manner, then $K$ is highly Kummer-faithful if and only if
$K$ is  highly $\ell$-Kummer-faithful for any prime number $\ell$.
\end{remark}

\begin{proposition} 
\label{KFHKF}
Let $K$ be a Galois extension of a Kummer-faithful field. 
If $K$ is quasi-highly Kummer-faithful,
then $K$ is Kummer-faithful.
\end{proposition}

\begin{proof}
Note that,
for any finite extension $L$ of $K$, 
any prime number $\ell$ and any $g$-dimensional abelian variety $A$ over $L$, 
there exist 
$G_L$-equivalent isomorphisms  
$\mbb{Q}_{\ell}(-1)\simeq H^2_{\rm \acute{e}t}(\mbb{P}^1_{\overline{K}},\mbb{Q}_{\ell})$
and 
$V_{\ell}(A)\simeq (H^1_{\rm \acute{e}t}(A_{\overline{K}},\mbb{Q}_{\ell}))^\vee$.
Hence the result follows from Propositions \ref{div} and \ref{KFreduction}.
\end{proof}

\begin{proposition} 
\label{subptorfin}
If $A$ is a semi-abelian variety over a sub $p$-adic field $K$,
then $A(K)_{\mrm{tor}}$ is finite.
\end{proposition}

\begin{proof}
We may assume that $K$ is a finite extension of $K_0=\mbb{Q}_p(t_1,\dots ,t_r)$
where $t_1,..,t_r$ are transcendental basis of $K$ over $\mbb{Q}_p$. 
The torsion part of $K$-rational points of 
any torus over $K$ is finite.
Thus it suffices to prove the proposition in the case where $A$ is an abelian variety.
Let $A_0:=\mrm{Res}_{K/K_0}(A)$ be the Weil restriction.
Then $A_0$ is an abelian variety over $K_0$ and  we have $A_0(K_0)=A(K)$.
By the Lang-N\'eron theorem (\cite{Co}, \cite{LN}), 
there exist an abelian variety $A'_0$ over $\mbb{Q}_p$ and 
an injection 
$\iota\colon A'_0(\mbb{Q}_p)\hookrightarrow A_0(K_0)$
with the property that 
$A_0(K_0)/\iota (A'_0(\mbb{Q}_p))$ is a finitely generated abelian group.
Since $A'_0(\mbb{Q}_p)_{\mrm{tor}}$ is finite (cf. \cite{Ma}),
we obtain that $A_0(K_0)_{\mrm{tor}}=A(K)_{\mrm{tor}}$ is finite.
\end{proof}

\bn
{\bf Summary.} 
Let $K$ be a Galois extension of a Kummer-faithful  field.
For such $K$, we consider the following conditions: 

\begin{description}
\item[(Sub${}_p$)]  $K$ is sub $p$-adic.
\item[(Tor)${}_{\mrm{fin.}}$]  For any finite extension $L$ of $K$ and any semi-abelian variety $A$ over $L$,
it holds that $A(L)_{\mrm{tor}}$ is finite.
\item[(Tor)${}_{\mrm{loc.fin.}}$] For any finite extension $L$ of $K$, any semi-abelian variety $A$ over $L$
and any prime number $\ell$,
it holds that $A(L)[\ell^{\infty}]$ is finite.
\item[(KF)] $K$ is Kummer-faithful. 
\item[(QHKF)] $K$ is quasi-highly Kummer-faithful. 
\item[(HKF)] $K$ is highly Kummer-faithful. 
\end{description}

\noindent
Then we have the following diagram of logical relations:
$$
\displaystyle \xymatrix{
& {\bf (KF)}
\ar@{<=>}^{{\rm Prop. \ref{div}}\qquad }[rr] 
& &  
{\bf (Tor){}_{\mrm{loc.fin.}}}
\\
{\bf (HKF)} \ar@{=>}[r]
&
{\bf (QHKF)} \ar@{=>}^{{\rm Prop. \ref{KFHKF}}}[u]
& 
({\bf Sub}{}_p) \ar@{=>}_{\cite{Mo}}[lu] \ar@{=>}^{{\rm Prop. \ref{subptorfin}}}[r]
& 
{\bf (Tor){}_{\mrm{fin.}}}  \ar@{=>}_{{\rm trivial}}[u]
}
$$

It seems  natural  to consider the relationship between sub $p$-adic fields and 
(quasi-)highly Kummer-faithful fields.
First, we note that  {\it any sub $p$-adic field is  not highly Kummer faithful}. 
To see this, it is enough to show that $\mbb{Q}_p$ is not highly  Kummer faithful.
Let $E$ be a Tate curve over $\mbb{Q}_p$.
Since we have a natural exact sequence $0\to \mbb{Q}_p(1)\to V_p(E)\to \mbb{Q}_p\to 0$
of representations of $G_{\mbb{Q}_p}$,
we obtain
$(H^1_{\rm \acute{e}t}(E_{\overline{\mbb{Q}}_p},\mbb{Q}_p(1))_{G_{\mbb{Q}_p}})^{\vee}
\simeq 
(H^1_{\rm \acute{e}t}(E_{\overline{\mbb{Q}}_p},\mbb{Q}_p(1))^{\vee})^{G_{\mbb{Q}_p}}
\simeq 
V_p(E)(-1)^{G_{\mbb{Q}_p}}\not=0$.
Thus we are done. 
Now we can ask the following weaker question:

\begin{question}
Are sub $p$-adic fields quasi-highly Kummer-faithful?
\end{question}

Under the ``$\ell$-adic'' and ``$p$-adic'' monodromy weight conjectures 
(Conj.\ 4.1 and 5.1 of \cite{Ja}),
we can show the following.

\begin{proposition}
Assume the monodromy weight conjectures. 
Then any $p$-adic field is quasi-highly Kummer-faithful. 
\end{proposition}

\begin{proof}
Let $K$ be a $p$-adic field, $L$ a finite extension of $K$ 
and $X$ a proper smooth variety over $L$ of dimension $g$.
Let $\ell$ be any prime number (including the case $\ell=p$).
Since we have  
$(H^i_{\rm \acute{e}t}(X_{\overline{K}},\mbb{Q}_{\ell})_{G_L})^{\vee}
\simeq 
(H^i_{\rm \acute{e}t}(X_{\overline{K}},\mbb{Q}_{\ell})^{\vee})^{G_L}
\simeq 
H^{2g-i}_{\rm \acute{e}t}(X_{\overline{K}},\mbb{Q}_{\ell}(g))^{G_L}$,
 for our purpose,
it suffices to check the vanishing of 
$H^{2g-i}_{\rm \acute{e}t}(X_{\overline{K}},\mbb{Q}_{\ell}(g))^{G_L}$. 
Since we have  $g\notin [\mrm{max}\{0,g-i\}, (2g-i)/2]$ for $i>0$,
it follows from Corollaries 4.3 and 5.2 of \cite{Ja} 
(under the assumption of the monodromy weight conjectures)  
that the above subspace is zero for any prime number $\ell$.  
\end{proof}

\subsection{Criteria for high Kummer-faithfulness}

We give some criteria for high Kummer-faithfulness in terms of ramification theory. 

\begin{theorem} 
\label{Main:HKF}
Let $K$ be a Galois extension of a number field $k$. 
Assume that, for any finite extension $k'/k$,
the ramification index of the maximal abelian subextension of $Kk'/k'$
at any finite place of $k'$ is finite. 
Then $K$ is highly Kummer-faithful.

In particular,  any number field is highly Kummer faithful.
\end{theorem}

The theorem is a consequence 
of a general property of Galois representations (Lemma \ref{Main} below).
To state it, we recall the notion of Weil weights.

\begin{definition}
Let $k$ be a number field, $v$ a finite place of $k$ not above $\ell$, 
$V$ an $\ell$-adic representation of $G_k$
and $w$ a real number.
We say that $V$ has (pure) {\it Weil weight $w$ at $v$} if
\begin{itemize}
\item[(a)] $V$ is unramified at $v$, 
\item[(b)] coefficients of the characteristic polynomial 
$\mrm{det}(T-\mrm{Frob}_v \mid V)$ are algebraic numbers, and 
\item[(c)] for any root $\alpha$ of $\mrm{det}(T-\mrm{Frob}_v \mid V)$ and any embedding 
$\iota\colon \overline{\mbb{Q}}\hookrightarrow \mbb{C}$,
it holds that $|\iota (\alpha)|=q_v^{w/2}$. 
Here, $q_v$ is the order of the residue field of $k$ at $v$.
\end{itemize}
\end{definition}

For a proper smooth 
variety $X$ over a number field $k$ which has good reduction at a finite place $v$ of $k$, 
the Weil conjecture proved by Deligne \cite{De1, De2} implies that 
$
H^i_{\rm \acute{e}t}(X_{\overline{k}},\mbb{Q}_{\ell}(r))
$ 
has  Weil weight $i-2r$ at $v$. 

\begin{lemma} 
\label{Main}
Let $K$ be a Galois extension of a number field $k$ 
and $\ell$ a prime number.
Let $M/k$ be the maximal abelian subextension of $K/k$.  
Assume that the ramification index of the extension $M/k$ 
at any finite place of $k$ above $\ell$ is finite. 
Then, for any  $\mbb{Q}_{\ell}$-representation $V$ of $G_k$
with non-zero Weil weight at some finite place of $k$ not above $\ell$, we have $V^{G_K}=0$.
\end{lemma}

\begin{proof}
Assume that $V^{G_K}\not=0$. Then
$V^{G_K}$ is a non-zero $G_k$-stable submodule of $V$
since $K$ is a Galois extension of  $k$.
Let $\psi\colon G_k\to \Ql^\times$
be the character such that  $\mrm{det}\ V^{\GK}=\Ql(\psi)$.
We denote by  $k(\psi)/k$  the splitting field extension of $\psi$.
It follows from class field theory 
that the extension  $k(\psi)/k$
is potentially unramified at all finite places of $k$ not above $\ell$
and is unramified at all but finitely many places of $k$. 
Since $k(\psi)$ is an abelian subextension of $K/k$, 
we have $k(\psi)\subset M$. Hence 
the assumption on the extension $M/k$ implies that 
the ramification degree of $k(\psi)/k$ at any finite place of $k$ 
above $\ell$ is finite. 
Then class field theory implies that 
the extension $k(\psi)/k$ is finite, and so 
$\psi(G_k)$ is finite. 
This contradicts the assumption that
$V$  (and hence  $\det V=\Ql(\psi)$  also) has non-zero Weil weight 
at some finite place of  $k$. 
\end{proof}

\begin{proof}[Proof of Theorem \ref{Main:HKF}]
Let $L$ be a finite extension of $K$, 
$X$ a proper smooth variety over $L$,
$\ell$ a prime number and
$i, r$ be integers such that $i\not = 2r$.
Put $V=H^i_{\rm \acute{e}t}(X_{\overline{K}},\mbb{Q}_{\ell}(r))$.
We take a finite extension  $k'$ of $k$ so that 
$L=Kk'$ and $X$ is defined over $k'$. 
Then, $V$ is a $\mbb{Q}_{\ell}$-representation of $G_{k'}$ 
with non-zero Weil weight $i-2r$ at any finite place $v$ 
of $k$ not above $\ell$ such that $X$ has good reduction at $v$.
It follows from  the assumption on $K$, $k$ and Lemma \ref{Main}, that  
$V^{G_L}=0$.
\end{proof}

\begin{definition} 
\label{KFandHKF} 
(1)  Let $K$ be an algebraic extension of a $p$-adic field $k$
and $\tilde{K}$ the Galois closure of  $K/k$. 
Put $G=\mrm{Gal}(\tilde{K}/k)$. 
We say that 
the extension $K/k$ has {\it finite maximal ramification break}
if $G^c$ is trivial\footnote{
We follow Serre's convention (\cite{Se1},  Chap.\ IV, \S 3) for the 
$i$-th upper ramification subgroup $G^i$.
In particular, $G^{-1}=G$,
$G^i$ for $-1< i\le 0$ is the inertia subgroup of $G$, and 
$G^{0+}:=\bigcup_{i>0}G^i$ is the wild inertia subgroup of $G$.
}
for  $c$  large enough. 
In  the case where  $K$  is an abelian extension of  $k$, 
this is equivalent to the condition that 
the ramification index of  $K/k$  is finite 
by local class field theory.

\noindent
(2) Let $K$ be a Galois extension of  a number field $k$.
For any finite place $v$ of $k$, 
We say that 
the extension $K/k$ has a {\it finite maximal ramification break at $v$}
if the extension  $K_w/k_v$ has finite maximal ramification break
where $w$ is a finite place of $K$ above $v$.
Note that  this definition does not depend on the choice of $w$.
Note also that, if $K$ is an abelian extension of  a number field $k$,
$K/k$ has  finite maximal ramification break at a finite place  $v$ of $k$
if and only if the ramification index of $K/k$  at $v$ is finite.

\noindent
(3) Let $K$ be a Galois extension of  a number field $k$.
We say that 
the extension $K/k$ has {\it finite maximal ramification break everywhere}
if  it has finite maximal ramification break at any finite place $v$ of $k$.
\end{definition}

Note that the property 
``having finite maximal ramification break'' is unchanged 
by finite extension of the base field  $k$. 
Hence the following corollary follows from Theorem \ref{Main:HKF}: 

\begin{corollary} 
\label{Main:HKF:cor}
Let  $K$  be a Galois extension of a number field  $k$. 
Assume that the extension  $K/k$  has 
finite maximal ramification break everywhere.
Then  $K$  is  highly Kummer-faithful.
\end{corollary}


\section{Examples}

In this section, 
we give some examples of highly Kummer-faithful 
fields which are infinite algebraic extensions of number fields. 
To check high Kummer-faithfulness,
we often use 
Theorem  \ref{Main:HKF}, 
Corollary  \ref{Main:HKF:cor} and the following lemma: 

\begin{lemma} 
\label{Main:HKF:lem}
Let $k$ be a $p$-adic field and $\{ K_i \}_{i\in I}$ 
a family of Galois extensions of $k$.
Let $K_I$ be the composite field of the $K_i$'s for all  $i\in I$.
Then, for a real number  $c\ge -1$, 
$\mrm{Gal}(K_I/k)^c$ is trivial if and only if 
$\mrm{Gal}(K_i/k)^c$ is  trivial for any $i\in I$.
\end{lemma}

\begin{proof}
By \cite[Chap.\ IV, Prop.\ 14]{Se1}, the natural surjection
$\mrm{Gal}(K_I/k)\twoheadrightarrow \mrm{Gal}(K_i/k)$
induces a surjection 
$\mrm{Gal}(K_I/k)^c\twoheadrightarrow \mrm{Gal}(K_i/k)^c$.
Furthermore, we have a natural injection 
$\mrm{Gal}(K_I/k)\hookrightarrow \prod_{i\in I} \mrm{Gal}(K_i/k)$.
The lemma immediately follows from these maps.
\end{proof}

The following is an immediate consequence of 
Corollary \ref{Main:HKF:cor} and Lemma \ref{Main:HKF:lem}.

\begin{corollary}
Let  $k$  be a number field.

\noindent
{\rm (1)} Let $d>0$ be an integer.
Then the composite field of all finite extensions of  $k$  
of degree $\le d$ is highly Kummer faithful.

\noindent
{\rm (2)} 
Let $\{k_i\}_{i\in I}$ be a family of finite extensions of  $k$.
Assume that the discriminants of $k_i/k$ are 
prime to each other.
Then the composite field of  $k_i$'s for all  $i\in I$ is highly Kummer faithful.
\end{corollary}

\subsection{Construction of highly Kummer-faithful fields from semi-abelian varieties}
Let $k$ be a number field.
Let $g>0$ be an integer and $\mbf{m}:=(m_p)_p$
a family of non-negative integers where $p$ ranges over the prime numbers. 
Let $k_{g,\mbf{m}}$ be the extension field of $k$ obtained by adjoining all coordinates of 
elements of $B[p^{m_p}]$ for all semi-abelian varieties $B$ over $k$ of dimension at most $g$ 
and all prime numbers $p$.

\begin{theorem}
\label{HKF:ex1}
{\rm (1)} The extension $k_{g,\mbf{m}}/k$ has 
finite maximal ramification break everywhere.

\noindent
{\rm (2)} The field $k_{g,\mbf{m}}$ is highly Kummer-faithful.
\end{theorem}

\begin{proof}
(2) follows from (1) by Corollary  \ref{Main:HKF:cor}. 
To prove (1), take any finite place $v$ of $k$.
By Lemma \ref{Main:HKF:lem}, it suffices to show 
that there exists an integer $c_v>0$ with the property that 
$\mrm{Gal}(k_v(B[p^{m_p}])/k_v)^{c_v}$ is trivial 
for any semi-abelian variety $B$ over $k$ of dimension at most $g$ 
and any prime number $p$.

Let ${B_v}_{/k_v}$ be a semi-abelian variety of dimension at most $g$ 
which is an extension of 
an abelian variety ${A_v}_{/k_v}$ by a torus ${T_v}_{/k_v}$.
By Raynaud's criterion of semi-stable reduction \cite[Prop.\ 4.7]{Gr},
the abelian variety $A_v$ has semi-stable reduction over $k_v(A_v[12])$.
We clearly have 
$[k_v(A_v[12]):k_v]\le \#\mrm{GL}_{2g}(\mbb{Z}/12\mbb{Z})=:c(g)$.
On the other hand, set $X(T):=\mrm{Hom}_{\overline{k}_v}(T_v,\mbf{G}_m)$
and $r:=\mrm{dim} T_v$.
Then $X(T)$ is a free $\mbb{Z}$-module of rank $r$ and $G_{k_v}$
acts on $X(T)$. If we denote by $k_{v,0}/k_v$ the splitting field of 
this $G_{k_v}$ action on $X(T)$, then the extension $k_{v,0}/k_v$ is finite and Galois, 
the torus $T_v$ splits over $k_{v,0}$
and we have an inclusion 
$\mrm{Gal}(k_{v,0}/k_v)\hookrightarrow \mrm{GL}_r(\mbb{Z})$. 
It is shown by Minkowski\footnote{
More easily, 
the existence of $c'(g)$ immediately follows from the fact that 
the kernel of the projection 
$\mrm{GL}_r(\mbb{Z})\to \mrm{GL}_r(\mbb{Z}/3\mbb{Z})$ is torsion free.
} 
\cite{Mi} that  
there exists an integer $c'(g)>0$, depending only on $g$,
such that the order of any torsion subgroup of $\mrm{GL}_g(\mbb{Z})$
is bounded by $c'(g)$.
Thus we have $[k_{v,0}:k_v]\le c'(g)$.
Now we denote by $k_v'$ the composite field of 
all Galois extensions of $k_v$ of degree at most $\mrm{max}\{c(g),c'(g)\}$.
Since there exist only finitely many finite extensions of given degree
over a $p$-adic field, 
we see that $k_v'$ is a finite Galois extension of $k_v$.
We note that it follows from the above observation that, 
 for any semi-abelian variety ${B_v}_{/k_v}$
of dimension at most $g$, 
$B_v\otimes_{k_v} k_v'$ is an extension of 
an abelian variety with semi-stable reduction by a split torus.
Take a real number $c_v'>0$ so that $\mrm{Gal}(k_v'/k_v)^{c_v'}$ is trivial.

Let ${B_v}_{/k_v}$ be a semi-abelian variety of dimension at most $g$.
If we denote by $p_v$ the prime number below $v$,
then we have 
$[k_v(B[p_v^{m_{p_v}}]):k_v]\le 
\#\mrm{GL}_{2g}(\mbb{Z}/p_v^{m_{p_v}}\mbb{Z})=:c(g,\mbf{m},v)$.
We denote by $k_v''$ the composite field of 
all Galois extensions of $k_v$ of degree at most $c(g,\mbf{m},v)$.
Take a real number $c_v''>0$ so that $\mrm{Gal}(k_v''/k_v)^{c''_v}$ is trivial.

Now we set $c_v:=\mrm{max}\{c_v',c_v'' \}$. We show that 
this  $c_v$  has the desired property.
Let $B_{/k}$ be a semi-abelian variety of dimension at most $g$ 
which is an extension of  an abelian variety $A_{/k}$ by a torus $T_{/k}$.
For the prime $p$ below $v$, the inequality $c_v\ge c_v''$ implies that  
$\mrm{Gal}(k_v(B[p^{m_p}])/k_v)^{c_v}$ is trivial.
Next we take any prime $p$ which is not below $v$.
Let $k_v'$ be the  finite Galois extension of $k_v$ defined above. 
Since $T\otimes_k k_v'$ is a split torus and 
$A\otimes_k k_v'$ has semi-stable reduction, we see that 
the inertia subgroup $I_v'$ of  $G_{k_v'}$  acts unipotently on  $V_{p}(B)$.
In particular, $\rho_{B,p}(I_v')$ is a pro-$p$ group 
where $\rho_{B,p}\colon G_k\to \mrm{GL}_{\mbb{Q}_p}(V_p(B))$ 
is the  continuous homomorphism 
obtained by the $G_k$-action on $V_p(B)$. 
On the other hand, we have $G_{k_v}^{c_v}\subset G_{k_v}^{c_v'}\subset I'_{v}$
by definition of $c_v$ and $c_v'$, 
and $G_{k_v}^{c_v'}$ is a pro-$p_v$ group since  $c_v'>0$.
Hence  $\rho_{B,p}(G_{k_v}^{c_v})$  is trivial,
which implies that $\mrm{Gal}(k_v(B[p^{m_p}])/k_v)^{c_v}$ is trivial as desired. 
\end{proof}

\begin{remark}
We recall that any sub-$p$-adic field is Kummer-faithful
(cf. \cite[Rem.\ 1.5.4]{Mo}). 
If  $m_p>0$  for infinitely many  $p$, 
our field  $k_{g,\mbf{m}}$  above gives an example of a field that is 
not sub-$p$-adic but Kummer-faithful.
Indeed, let  
$M$  be the subfield of  $\kgm$  
obtained by adjoining to  $\mbb{Q}$  all $p$th roots of unity 
for all primes  $p$  such that  $m_p>0$. 
If  $k_{g,\mbf{m}}$ is sub-$p$-adic, then  
$M$  is also sub-$p$-adic. 
However, this is impossible since the residue field of  $M$
at any finite place is infinite. 
\end{remark}

\subsection{High Kummer-faithfulness of abelian extensions over $\mbb{Q}$}
Theorem \ref{Main:HKF} says that 
the condition of finite ramification is 
sufficient for a Galois extension  $K$  of a number field  $k$  
to be highly Kummer-faithful. 
(Thus, for example, if  $K/k$  is a class field tower, 
then  $K$  is highly Kummer-faithful.)
It is natural to ask if it is also necessary. 
Unfortunately, the necessity does not hold in general; 
see \S \ref{Sect:inf.ram.}.
If, however, $K$  is an {\it abelian extension of} $\Q$, then 
Proposition \ref{Prop:abelian} below shows the necessity. 

Now we set $\mu_{(2)}:=\mu_{4}$,\ $\mu_{(\ell)}:=\mu_{\ell}$ for odd primes $\ell$
and set 
$\mu:=\bigcup_{\ell} \mu_{(\ell)}$.\
Following \cite[Def.\ 1.5]{Mo}, 
we say that a perfect field $K$ is {\it torally Kummer-faithful}
if, for every finite extension $L$ of $K$ and every torus $T$ over $L$,
it holds that 
$$
T(L)_{\rm div} =0.
$$ 
By definition, Kummer-faithful fields are torally Kummer faithful.

\begin{proposition} 
\label{Prop:abelian}
Let $K$ be an abelian extension of $\mbb{Q}$. 
Then the following are equivalent.
\begin{itemize}
\item[{\rm (1)}] The ramification index of  the extension 
$K/\mbb{Q}$ is finite at any prime number. 
\item[{\rm (2)}] $K$ is highly Kummer-faithful.
\item[{\rm (3)}] $K$ is Kummer-faithful.
\item[{\rm (4)}] $K$ is torally Kummer-faithful. 
\item[{\rm (5)}] The $\ell$-adic cyclotomic character  
$\chi_\ell:\GK\to\Zl^\times$  has open image for any prime number  $\ell$.
\item[{\rm (6)}]
$K$ is a subfield of $\mbb{Q}(\mu_{p^{n_p}}\mid \mbox{$p$ : prime})$
for some family of positive integers $(n_p)_p$.
\end{itemize}
\end{proposition}
\begin{proof}
We know already that
(1) $\Rightarrow$ (2) $\Rightarrow$ (3) $\Rightarrow$  (4) 
(see Corollary \ref{Main:HKF:cor}, Proposition \ref{KFHKF}).
The implication  (4) $\Rightarrow$ (5) is easy to see, 
as noted in Remark 1.5.1 of \cite{Mo}. 
The equivalence of (1) and (6) is clear from class field theory.
It remains to prove  (5) $\Rightarrow$ (1). 
Let  $K(\mulinf)$  be the maximal $\ell$-power cyclotomic extension of  $K$. 
Since  $K/\Q$  is abelian, 
by class field theory, 
the inertia group of  $K(\mulinf)/\Q$  at  $\ell$  is isomorphic to  $\Zl^\times$, 
and that of  $K(\mulinf)/K$  at a place above  $\ell$  is
identified with an open subgroup of
$\Im(\chi_\ell)$. 
Thus if  $\Im(\chi_\ell)$  is open in  $\Zl^\times$, 
the ramification index of  $K/\Q$  at  $\ell$  is finite.
\end{proof}

\subsection{Kummer-faithful fields with infinite ramification}\label{Sect:inf.ram.}

The finite ramification condition is in general not necessary for a 
Galois extension of a number field to be Kummer-faithful. 
In this subsection, 
we apply the main theorem of \cite{Oz} to produce 
Kummer-faithful fields with infinite ramification. 

\begin{theorem}[{\cite[Theorem 1.1]{Oz}}]
\label{Oz}
Let  $k$  be a Galois extension of  $\Qp$ of degree  $d$.
Let  $\pi$  be a uniformizer of  $k$  and  $q$  the order of the residue field of  $k$.
Denote by  $k_{\pi}/k$  the 
Lubin-Tate extension associated with  $\pi$.
Assume that neither of the following two conditions hold.
\begin{itemize}
\item[{\rm (a)}] $q^{-1}N_{k/\mbb{Q}_p} (\pi)$ is a root of unity, 
where  $N_{k/\mbb{Q}_p}$  denotes the norm map of  $k/\mbb{Q}_p$. 
\item[{\rm (b)}] $N_{k/\mbb{Q}_p} (\pi)$ is a $q$-Weil integer 
of weight $d/t$ for some integer  $1\le t \le d$.
\end{itemize}
Then $k_{\pi}$ is Kummer-faithful.
\end{theorem}

Note that the theorem is proved
in a more general setting  in \cite{Oz}. 
However, we content ourselves with the above version 
to keep the notation simple in the following construction 
of Kummer-faithful fields,

Now let $k$ be a number field. 
Let  $p$  be a prime number which splits completely in  $k$, so that  
we identify the completion of  $k$  at a place above  $p$  with  $\Qp$. 
Let  $\pi$  be an element of  $k$  which gives rise to a uniformizer of  $\Qp$. 
Set $f(x):=\pi x+x^p\in k[x]$. 
Let $K$ be the extension field of $k$ obtained by adjoining 
all roots of $f^n(x)=0$  for all $n>0$
(here $f^n$ is the $n$-th composite of $f$). 
We apply the theorem with  $d=1$; 
if the complex absolute value of  $\pi$  with respect to some embedding 
$\overline{\mbb{Q}}\hookrightarrow \mbb{C}$  
is not equal to either 
$p$ or $\sqrt{p}$, 
then $K$ is Kummer-faithful. 
Indeed, $K$ is a subfield of the Lubin-Tate extension $k_{\pi}$
of $\Qp$ defined by $f(x)$ 
and Theorem \ref{Oz} shows that $k_{\pi}$  is Kummer-faithful. 

\subsection{A remark on the composite of Kummer-faithful fields}

The condition ``having finite maximal ramification break'' 
is closed under the composition of fields 
(Lemma \ref{Main:HKF:lem}). 
In view of Corollary \ref{Main:HKF:cor}, 
one may wonder if the composite field of 
Kummer-faithful Galois extensions of a number field is again 
Kummer-faithful. 
At present, we do not know if this is true or not. 
In the local case, however, this is {\it not} true, as the following example shows:

\begin{example}
For uniformizers $\pi_1$ and $\pi_2$ of 
$(k=)\mbb{Q}_p$, we consider the following properties.
\begin{itemize}
\item[(1)] 
The Lubin-Tate extensions $k_{\pi_1}$ and $k_{\pi_2}$ of $k:=\mbb{Q}_p$ 
associated with $\pi_1$
and $\pi_2$ respectively, 
are Kummer-faithful.
\item[(2)] 
The intersection  $k_{\pi_1}\cap k_{\pi_2}$  is a finite extension of $\mbb{Q}_p$.
\end{itemize}

We claim that, if  $\pi_1$  and  $\pi_2$  satisfy the condition (2), 
then the composite field  $K$  of  $k_{\pi_1}$  and  $k_{\pi_2}$  
is not Kummer-faithful.
Indeed, 
let $M$ be the maximal abelian pro-$p$  extension of $\Qp$.
By the condition (2), the maximal pro-$p$ subextension $L/\Qp$ of 
$K/\Qp$  is a  $\Z_p^2$-extension and the extension $K/L$ is finite.
In particular, it follows from local class field theory that 
$M$ is a finite extension of $L$.
On the other hand, $M$ is not Kummer-faithful
since  $M(\mu_p)$ contains all $p$-power roots of unity. 
Then it follows that $L$ and $K$ are not Kummer-faithful.

Now we show that there exists a pair 
$(\pi_1,\pi_2)$  which satisfies both conditions (1) and (2). 
We first note that, by Theorem \ref{Oz},  
the condition (1) holds if  $p^{-1}\pi_i$  is not a root of unity
and $\pi_i$ is not a $p$-Weil integer of weight 1 for $i=1,2$. 
Hence, to show the existence of a desired pair $(\pi_1,\pi_2)$,
it is enough to prove that  
the condition (2) holds if $\pi_1\pi_2^{-1}$ is not a root of unity.
Assume that $k_{\pi_1}\cap k_{\pi_2}$ is an infinite extension of $\mbb{Q}_p$.
Let $k^1_{\pi_i}$ be the subextension of $k_{\pi_i}/\Qp$ of degree $p-1$
and set $K_1:=k_{\pi_1}k^1_{\pi_2}$ and $K_2:=k_{\pi_2}k^1_{\pi_1}$.
Then extensions $K_1/k^1_{\pi_1}k^1_{\pi_2}$ and $K_2/k^1_{\pi_1}k^1_{\pi_2}$
are $\mbb{Z}_p$-extensions.
By the assumption, we have that $K_1\cap K_2$ 
is an infinite degree extension of $k^1_{\pi_1}k^1_{\pi_2}$.
This implies the equality $K_1\cap K_2=K_1=K_2$.
Now we denote by $\chi_{\pi_i}\colon G_{\mbb{Q}_p}\to \mbb{Z}_p^{\times}$
the Lubin-Tate character associated with $\pi_i$ (cf.\ \cite[Chap.\ III, A4]{Se2}). 
If we regard $\chi_{\pi_i}$ as a continuous 
character $\mbb{Q}_p^{\times}\to \mbb{Q}_p^{\times}$ 
by local class field theory,
then $\chi_{\pi_i}$ is characterized by the property that $\chi_{\pi_i}(\pi_i)=1$
and $\chi_{\pi_i}(u)=u^{-1}$ for any $u\in \mbb{Z}_p^{\times}$.
Note that the extension field of $\mbb{Q}_p$ which
corresponds to the kernel of $\chi_{\pi_i}$ is $k_{\pi_i}$. 
We denote by $\eta=\chi_{\pi_1}\chi^{-1}_{\pi_2}$ 
and also denote by  $\mbb{Q}_p(\eta)$ 
the extension field of $\mbb{Q}_p$ which
corresponds to the kernel of $\eta$. 
Then $\mbb{Q}_p(\eta)$ is an unramified 
subextension of $k_{\pi_1}k_{\pi_2}/\mbb{Q}_p$, and 
it follows from the equality $k_{\pi_1}k_{\pi_2}=K_1K_2=K_1$ that 
the residue field of $k_{\pi_1}k_{\pi_2}$ is finite.
Hence we obtain the fact that 
 $\mbb{Q}_p(\eta)$  is a finite (unramified) extension of $\mbb{Q}_p$,
that is, $\eta(G_{\mbb{Q}_p})$ is finite. 
Thus we have $(\chi_{\pi_1}\chi^{-1}_{\pi_2})^j=1$ for some integer $j>0$.
Since we have $\chi_{\pi_1}(\pi_2)\chi^{-1}_{\pi_2}(\pi_2)
=\chi_{\pi_1}(\pi_1\cdot (\pi_1^{-1}\pi_2))\chi^{-1}_{\pi_2}(\pi_2)
=\pi_1\pi_2^{-1}$, it follows 
that $\pi_1\pi_2^{-1}$ is a root of unity as desired. 
\end{example}

\subsection{Construction of highly Kummer-faithful fields via integral $p$-adic Hodge theory}
Let $k$ be a number field.
Let $\mcal{L}=(\mcal{L}_v)_{v}$ be a collection of open normal subgroup $\mcal{L}_v$ of $I_v$
for each finite place $v$ of $k$.
Let $V$ be a geometric $\ell$-adic representation of $G_k$ (in the sense of \cite{FM}).
We say that the {\it inertial level of $V$ is bounded by $\mcal{L}$}
if $V|_{\mcal{L}_v}$ is semi-stable at every $v$. 
Let $h\ge 0$ be an integer. 
We say that the {\it length of Hodge-Tate weights of $V$ is bounded by $h$}
if, for each $v$ above $\ell$, there exists integers $a_v\le b_v$ such that 
$b_v-a_v\le h$ and $\mrm{HT}_v(V)\subset [a_v,b_v]$.
Here, $\mrm{HT}_v(V)$ is the set of  Hodge-Tate weights of $V$ at $v$.
(We normalize the notion of Hodge-Tate weights so that the Hodge-Tate weight 
of cyclotomic character is one.)
For a prime number $\ell$, we denote by $\mcal{R}^{\ell}_{\mcal{L},h}(G_k)$ 
the set of  geometric $\ell$-adic representations of $G_k$
whose inertial level are bounded by $\mcal{L}$
and the length of Hodge-Tate weights are bounded by $h$.

\begin{definition}
Let $T$ be a  torsion $\mbb{Z}_{\ell}$-representation of $G_k$.
We say that {\it $T$ comes from
$\mcal{R}^{\ell}_{\mcal{L},h}(G_k)$}
if there exist $V\in \mcal{R}^{\ell}_{\mcal{L},h}(G_k)$ and 
$G_k$-stable $\mbb{Z}_{\ell}$-lattices $L\subset L'$ in $V$
such that $T\simeq L'/L$.
\end{definition}

Let $\mbf{m}:=(m_p)_p$ be
a family of non-negative integers where $p$ ranges over the prime numbers. 
We denote by $k_{\mcal{L},h,\mbf{m}}$ the extension field of $k$ obtained 
by adjoining all splitting fields of all 
torsion $\mbb{Z}_{\ell}$-representations $T$ coming from
$\mcal{R}^{\ell}_{\mcal{L},h}(G_k)$ with  $\ell^{m_{\ell}}T=0$ 
for all prime numbers $\ell$. 

\begin{theorem}
\label{geom}
{\rm (1)} 
The extension $k_{\mcal{L},h,\mbf{m}}/k$ 
has finite maximal ramification break everywhere.

\noindent
{\rm (2)}
The field $k_{\mcal{L},h,\mbf{m}}$ is highly Kummer-faithful.
\end{theorem}

\begin{proof}
(2) follows from (1) by  Corollary  \ref{Main:HKF:cor}. 
To prove (1), 
we first reduce ourselves to the case where  $\sqrt{-1}\in k$. 
Put $k'=k(\sqrt{-1})$. For any finite place $v'$ of $k'$,
set $\mcal{L}'_{v'}:=G_{k'_{v'}}\cap \mcal{L}_v$ where $v$ is 
the finite place of $k$ below $v'$.
Then $\mcal{L}'_{v'}$ is an open subgroup of $I_{v'}$.
Put  $\mcal{L}'=(\mcal{L}'_{v'})_{v'}$ where $v'$ runs thorough all finite places of $k'$.
For any $V\in \mcal{R}^{\ell}_{\mcal{L},h}(G_k)$, we see
$V|_{G_{k'}}\in \mcal{R}^{\ell}_{\mcal{L}',h}(G_{k'})$.
Thus we have inclusions 
$k_{\mcal{L},h,\mbf{m}}\subset k_{\mcal{L},h,\mbf{m}}k' \subset k'_{\mcal{L}',h,\mbf{m}}$. 
It follows from this that, for the proof,
it is enough to show that  
$k'_{\mcal{L}',h,\mbf{m}}/k'$ has finite maximal ramification break everywhere.

In the rest of the proof, we assume $\sqrt{-1}\in k$. 
Take any finite place $v$ of $k$ and denote by $p$ the prime number below $v$. 
We take a finite Galois extension $k_v'/k_v$ 
such that the inertia subgroup $I_{k'_v}$ of $k'_v$ is contained in $\mcal{L}_v$.  
Let $c_v'>0$ be a real number such that $\mrm{Gal}(k'_v/k_v)^{c'_v}$ is trivial. 
Let $T$ be any 
torsion $\mbb{Z}_{\ell}$-representation coming from
$\mcal{R}^\ell_{\mcal{L},h}(G_k)$ with $\ell^{m_{\ell}}T=0$.
Thus there exist $V\in \mcal{R}^{\ell}_{\mcal{L},h}(G_k)$ and 
$G_k$-stable $\mbb{Z}_{\ell}$ lattices $L\subset L'$ in $V$
such that $T\simeq L'/L$. 
We denote by $k_v(T)$ the completion of the splitting field  $k(T)$ 
of the representation $T$ of $G_k$  at a place above $v$.  
For the proof, by Lemma \ref{Main:HKF:lem},
it suffices to show that there exists a real number $c_v>0$, depending only on 
$\mcal{L},h,m_p$ and $v$ (not on $T$), 
with the property that
$\mrm{Gal}(k_v(T)/k_v)^{c_v}$ is trivial.

First we consider the case $\ell\not= p$. 
Let $\rho\colon G_k\to GL_{\mbb{Q}_{\ell}}(V)$ be the continuous homomorphism 
describing the $G_k$-action on $V$.
It follows from the fact that $I_{k'_v}$-action on $V$  is  unipotent 
that the group $\rho(I_{k'_v})$ is pro-$\ell$. 
Since the group $G^{c'_v}_{k_v}$ is pro-$p$ and this is a subgroup of $I_{k'_v}$,
we obtain that $\rho(G^{c'_v}_{k_v})$ is trivial.
This in particular implies that $\mrm{Gal}(k_v(T)/k_v)^{c_v}$ is trivial. 

Next we consider the case $\ell=p$. 
Let $t>0$ be an integer so that $\chi^t_p \  \mrm{mod} \ p^{m_p}=1$ where $\chi_p\colon G_k\to \Zp^{\times}$
is the $p$-adic cyclotomic character.  
Since we have $T=T\otimes_{\mbb{Z}_p} \mbb{Z}_p(nt)$ for any integer $n$,
we may suppose that $\mrm{HT}_v(V)$ is contained in $[0,r]$ where $r:=t+h$.
Then 
$T|_{G_{k'_v}}$ is a quotient of two $G_{k_v'}$-stable lattices in a
semi-stable $\mbb{Q}_p$-representation with Hodge-Tate weights in $[0,r]$.
Since $T$ is killed by $p^{m_p}$, it follows from Theorem 1.1 of 
\cite{CL}\footnote{
Note that \cite{CL} studies integral $p$-adic Hodge theory for odd primes $p$.
However, the arguments given in \cite{CL} work also for the case $p=2$ under the condition that
$\hat{G}=G_{p^{\infty}} \rtimes H_K$, that is, $K_{p^{\infty}}\cap K_{\infty}=K$. 
This condition holds for any odd primes $p$. In the case $p=2$, the condition 
does not hold in general 
but it holds 
if $\sqrt{-1}\in K$. For this, see Proposition 4.1.5 of \cite{Li}. 
Our assumption  $\sqrt{-1}\in k$ at the beginning of this proof
is needed only here.
}
that there exists an integer $c_v''>0$, depending only on $k_v', m_p$ and $r$, 
such that 
$G_{k_v'}^{c_v''}$ acts on $T$ trivially (see also Theorem 1.1 of \cite{Ha}).  
Furthermore, we have
$$
G_{k_v'}^{c_v''}=G_{k_v'}\cap G_{k_v}^{\varphi_{k_v'/k_v}(c_v'')}
\supset G_{k_v}^{c_v'}\cap G_{k_v}^{\varphi_{k_v'/k_v}(c_v'')}
=G_{k_v}^{c_v}
$$
where $c_v:=\mrm{max}\{c_v',\varphi_{k_v'/k_v}(c_v'')\}$ and 
$\varphi_{k_v'/k_v}(u)$ is the Herbrand function for $k_v'/k_v$, that is, 
$$
\varphi_{k_v'/k_v}(u):=\int^u_{0} \frac{1}{(\mrm{Gal}(k_v'/k_v)_{0}:\mrm{Gal}(k_v'/k_v)_{t})}\ dt.
$$
Hence the action of $G_{k_v}^{c_v}$  on $T$  is trivial.
Therefore, we obtain that  
$\mrm{Gal}(k_v(T)/k_v)^{c_v}$ is trivial.
Note that  we may say that $c_v$ is determined 
by $\mcal{L},h,m_p$ and $v$. 

Now the constant $c_v$ satisfies the desired property.
\end{proof}

Let $d>0$ be an integer and  $\mbf{m}:=(m_p)_p$
a family of non-negative integers where $p$ ranges over the prime numbers. 
Let  $k_{d\mathchar`-\mrm{st},\mbf{m}}$ 
be the extension field of  $k$  obtained by adjoining all coordinates of 
elements of  $A[p^{m_p}]$  for all prime numbers  $p$  and all abelian varieties 
$A$ over $k$  which have semi-stable reduction everywhere
over an extension of  $k$  of degree at most  $d$. 
Note that we have no restriction on the dimensions of abelian varieties for this definition;
this is a remarkable difference between definitions of  
$k_{d\mathchar`-\mrm{st},\mbf{m}}$ and  $k_{g,\mbf{m}}$. 
As a consequence of Theorem \ref{geom}, 
we obtain the following.

\begin{corollary}
\label{HKF:st}
{\rm (1)}  The extension $k_{d\mathchar`-\mrm{st},\mbf{m}}/k$ has 
finite maximal ramification break everywhere.

\noindent
{\rm (2)} The field $k_{d\mathchar`-\mrm{st},\mbf{m}}$ is highly Kummer-faithful.
\end{corollary}

\begin{proof}
For each finite place $v$ of $k$,
let $k_v'$ be the composite field of all Galois extensions of $k_v$ of degree at most $d$.
Note that $k_v'$ is a finite extension of $k_v$.
Let $\mcal{L}_v$ ($\subset I_v$) be the inertia group of $k_v'$ 
and set $\mcal{L}:=(\mcal{L}_v)_v$.
Then the result follows from the fact that  
$k_{d\mathchar`-\mrm{st},\mbf{m}}$ is contained in 
$k_{\mcal{L},h,\mbf{m}}$ with $h=1$. 
\end{proof}

\end{document}